\documentclass[12pt,reqno]{amsart}

\usepackage{amsrefs}
\usepackage[colorlinks]{hyperref}
\hypersetup{citecolor=blue}
\usepackage{fullpage}
\usepackage{setspace}
\usepackage{graphicx}

\theoremstyle{plain}
\newtheorem{thm}{Theorem}[section]
\newtheorem{prp}[thm]{Proposition}
\newtheorem{cor}[thm]{Corollary}
\newtheorem{lem}[thm]{Lemma}
\newtheorem*{rem}{Remark}

\newcommand{\C}{\mathbb{C}}
\newcommand{\R}{\mathbb{R}}
\newcommand{\N}{\mathbb{N}}
\newcommand{\e}{\varepsilon}
\newcommand{\dist}{\mathrm{dist}}
\newcommand{\diag}{\mathrm{diag}}
\newcommand{\U}{\mathcal{U}}

\def\be{\begin{equation}}
\def\ee{\end{equation}}

\begin{document}

\title{Doubling chains on complements of algebraic hypersurfaces}

\author{Omer Friedland}
\address{Institut de Math\'ematiques de Jussieu, Universit\'e Pierre et Marie Curie (Paris 6), 4 Place Jussieu, 75252 Paris, France.}
\email{omer.friedland@imj-prg.fr}

\author{Yosef Yomdin}
\address{Department of Mathematics, The Weizmann Institute of Science, Rehovot 76100, Israel.}
\email{yosef.yomdin@weizmann.ac.il}

\thanks{Research of the second author was supported by the Israel Science Foundation, grant No. ISF 779/13}

\begin{abstract}
A doubling chart on an $n$-dimensional complex manifold $Y$ is a univalent analytic mapping $\psi:B_1\to Y$ of the unit ball in $\C^n$, which is extendible to the (say) four times larger concentric ball of $B_1$. A doubling covering of a compact set $G$ in $Y$ is its covering with images of doubling charts on $Y$. A doubling chain is a series of doubling charts with non-empty subsequent intersections. Doubling coverings (and doubling chains) provide, essentially, a conformally invariant version of Whitney's ball coverings of a domain $W\subset {\mathbb R}^n$, introduced in \cite{Whi} (compare \cite{Fri.Yom3}).

\smallskip

We study doubling chains in the complement $Y=\C^n\setminus H$ of a complex algebraic hypersurface $H$ of degree $d$ in $\C^n$, and provide information on their length and other properties. Our main result is that any two points $v_1,v_2$ in a distance $\delta$ from $H$ can be joined via a doubling chain in the complement $Y=\C^n\setminus H$ of length at most $c_1\log (\frac{c_2}{\delta})$ {\it with explicit constants $c_1,c_2$ depending only on $n$ and $d$}.

\smallskip

As a consequence, we obtain an upper bound on the Kobayashi distance in $Y$, and an upper bound for the constant in a doubling inequality for regular algebraic functions on $Y$. We also provide the corresponding lower bounds for the length of the doubling chains, through the doubling constant of specific functions on $Y$.
\end{abstract}

\maketitle

\section{Introduction}

Let us recall the definition of a doubling covering, as given in \cite{Fri.Yom3}. A doubling chart on an $n$-dimensional complex manifold $Y$ is a univalent analytic mapping $\psi:B_1\to Y$ of the unit ball $B_1$ in $\C^n$, which is extendible to a mapping $\tilde\psi_j:B_4\to Y$ regular and univalent in a neighborhood of $B_4$, where $B_4\subset\C^n$ is the four times larger concentric ball of $B_1$ (clearly $B_4$ may be replaced by $B_\gamma$ for $\gamma > 1$). A doubling covering $\U$ of a compact set $G$ in $Y$ is its finite covering with images of doubling charts on $Y$. The complexity $\kappa(\U)$ is the number of the charts in $\U$.

\smallskip

Doubling coverings provide a conformally invariant version of Whitney's ball coverings of \cite{Whi} (compare also \cite{Mar.Vuo,Zha.Kle.Suo.Vuo}). We refer to \cite{Bin.Nov,Fri.Yom3,Yom3,Yom4,Yom5} for a discussion of a connection of doubling coverings with ``smooth and analytic parametrizations'', and through them, to bounding entropy type invariants in smooth dynamics on one side, and to bounding density of rational points on analytic varieties in diophantine geometry, on the other.

\smallskip

In view of these connections, one of the most important problems related to double coverings $\U$ of $G$ in $Y$ is the explicit bounding of their complexity $\kappa(\U)$. Let us stress that in situations where the resolution of singularities works (algebraic, analytic, subanalytic, and some o-minimal settings), the mere existence of a finite doubling covering is immediate: we just double-cover a ``non-singular model'' of $(Y,G)$. However, $\kappa(\U)$ may blow up in families, and this presents a major obstacle in applications. Recently a serious progress was achieved in study of different types of smooth parametrizations (see \cite{Bin.Nov,Fri.Yom3,Yom5} and references therein), and, in particular, in bounding their complexity.

\smallskip

In the present paper we study ``doubling chains'' in $Y$, formed by doubling charts. We provide an upper bound for the length of doubling chains in $Y$ being the complement of an algebraic hypersurface in $\C$, and show its essential sharpness. Via the results of \cite{Fri.Yom3} this provides an explicit bound on the Kobayashi metric on this manifold, and on the constant in the doubling inequalities on $Y$. 

Notice that the Kobayashi distance on the complement of a union of hyperplanes in $\C P^n$ was studied in \cite{Che.Ere}, and a lower bound was obtained there. It would be interesting to compare the approaches of \cite{Che.Ere} and of the present paper.

\subsection{Main results} \label{Sec:main.res}

A doubling chain $Ch$ joining two points $v_1,v_2\in Y$ is a series of doubling charts $\psi_j$, $j = 1,\dots,l$, so that their images $U_j = \psi_j(B_1)$ (which will be also called charts) satisfy $U_j\cap U_{j+1}\ne \emptyset$, $j = 1,\dots, l-1$, and $v_1\in U_1, v_2\in U_l$. We denote by $l(Ch)$ the length of a chain $Ch$, that is, the number of its elements. For two neighboring charts $U_j$ and $U_{j+1}$ in a chain $Ch$ we define the intersection radius $\rho_j=\rho(U_j,U_{j+1})$ as the maximal radius $\rho > 0$ so that both $\psi^{-1}_j(U_j\cap U_{j+1})\subset B_1$ and $\psi^{-1}_{j+1}(U_j\cap U_{j+1})\subset B_1$ contain balls of radius $\rho$ (not necessarily concentric). We put $\rho(Ch)=\min_{j}\rho_j$ and call it the intersection radius of the chain $Ch$.

\smallskip

Doubling chains were introduced in \cite{Fri.Yom3}, as a part of a general construction of doubling coverings. In \cite{Fri.Yom3} we studied doubling coverings and doubling chains of a complex manifold, being a compact part of a non-singular level hypersurface $H=\{P=c\}$, where $P$ is a polynomial in $\C^n$ with non-degenerated critical points. It was shown in \cite{Fri.Yom3} that the complexity of a doubling covering of $H$ is of order $\log({1}/{\rho})$, where $\rho$ is the distance of $H$ from the singular set of $P$.

\smallskip

The main objective of the present paper is to study doubling chains in the {\it complement} $Y=\C^n\setminus H$ of a complex algebraic hypersurface $H$ in $\C^n$. Let us be more precise. Let $P(z) = \sum_{\alpha:|\alpha| \le d}a_\alpha z^\alpha$ be a complex polynomial of degree $d$ in $\C^n$ written in the usual multi-index notations $z = (z_1,\dots,z_n) \in \C^n$, $\alpha = (\alpha_1,\dots,\alpha_n)\in \N^n$, $|\alpha| = \sum_{i = 1}^n |\alpha_i|$ and $z^\alpha = z_1^{\alpha_1}\cdots z_n^{\alpha_n}$. We consider the zero hypersurface $H = \{P = 0\} \subset \C^n$. Later, without loss of generality, we'll assume that $P$ is normalized, i.e. $\|P\| :=\sum_{\alpha:|\alpha| \le d}|a_\alpha|=1$. For $\delta > 0$ we denote by $H^{\delta}$ the $\delta$-neighborhood of $H$ in $\C^n$, i.e. the set of points $z\in \C^n$ with $\dist(z,H) \le \delta$, and put $Q^{\delta} = Q\setminus H^{\delta}$, where $Q$ is the unit cube in $\C^n$.

\smallskip

The following theorem is our main result:

\begin{thm} \label{thm-main}
Let $P$ be a polynomial of degree $d$ in $\C^n$, and let $H=\{P=0\}$. Then for any $\delta$ with $0< \delta \le \rho(n,d)=\frac{1}{4(16(d+n))^n}$, and for any $v_1,v_2\in Q^{\delta}$ there exists a doubling chain $Ch$ in $Y=\C^n\setminus H$, joining $v_1$ and $v_2$, with the following properties:

\smallskip

\noindent 1. The length $l(Ch)$ of the chain $Ch$ satisfies
$$
l(Ch) \le 36d\log(180d/\delta)+1 .
$$

\smallskip

\noindent 2. The intersection radius $\rho(Ch)$ satisfies $\rho(Ch) \ge 2^{-d}/3$.

\smallskip

\noindent 3. The charts $U_j = \psi_j(B_1)$ in the chain $Ch$ are contained in $Q^{\bar \delta}$, with $\bar \delta:=c_{n,d}\delta^d$, where $c_{n,d}>0$ depends only on $n,d$. In particular, any two points $v_1,v_2\in Q^{\delta}$ belong to the same connected component of $Q^{\bar \delta}$.
\end{thm}

Let us stress the following important features of this result:

\smallskip

\noindent 1. Covering not the entire set $A$, but a complement to a $\delta$-neighborhood of a certain algebraic submanifold $\Sigma \subset A$ (typically, containing singularities of A), was a ``trick'' developed in \cite{Yom2}, already in the case of $C^k$-parametrizations. Ultimately, after Gromov's final version of a $C^k$-parametrizations theorem (\cite{Gro}), which provided a parametrization of the entire set $A$, with the number of charts bounded in terms of $n$ and $d$ only, this trick was not necessary any more in applications to a $C^k$-smooth dynamics. However, the complexity of an {\it analytic parametrization of $A$ depends indeed on the specific parameters of the polynomials defining $A$, and not only on their degree} (see \cite{Fri.Yom3,Yom4,Yom5}). Consequently, covering the set $A$, but a $\delta$-neighborhood of a certain algebraic submanifold in $A$, becomes a rather relevant, (and, presumably, unavoidable) part of the approach.

\smallskip

\noindent 2. In this context, the most important parameter becomes $\delta$, the size of the removed neighborhood. The expected bound of order $\log(\frac{1}{\delta})$ for the length of the chain is essential for all the expected applications. Moreover, the fact that the constants in the (logarithmic in $\delta$) bound of Theorem \ref{thm-main} depend only on $n$ and $d$, but not on the specific coefficients of the defining polynomial $P$, is critically important for the planned applications in analytic dynamics, and, presumably, also for possible applications in diophantine geometry.

\smallskip

The paper is organized as follows. In Section \ref{sec-proof} we prove the main result of this paper. As a consequence, we present some applications of this result. In Section \ref{sec-kob} we obtain an upper bound on the Kobayashi distance in $Y=\C^n\setminus H$, and in Section \ref{sec-double} we discuss doubling inequalities and their relations to doubling chains in the complement of algebraic hypersurfaces in $\C^n$. We obtain an upper bound on the doubling constant for regular algebraic functions on $Y$, and we provide also a lower bound for the length of the doubling chains, through the doubling constant of a specific function $f=\frac{1}{P}$ on $Y$.

\section{Proof of Theorem \ref{thm-main}} \label{sec-proof}

First we sketch the proof of Theorem \ref{thm-main} (the details are given below in Sections \ref{sec-ball}-\ref{sec-chain}). The idea is to join the two points $v_1,v_2\in Q^{\delta} = Q\setminus H^{\delta}$ by a complex straight line $L$, and consider the zeros set of the restriction of $P$ to $L$, which we denote by $Z=L\cap H=\{u_1,\dots,u_d\}$. We show that $v_1,v_2\in Q^\delta \cap L$ belong to the same connected component of $Q^{\delta'} \cap L$, with $\delta'=\frac{\delta}{10d}$. Then, we cover $Q^{\delta}\cap L$ with doubling disks $D^j$, using the ``Zigmund-Calderon'' covering construction of \cite{Fri.Yom3} (compare also \cite{Mar.Vuo}). Most important for the ultimate bounds in terms of $n,d$ only is the fact that the bound in the covering construction of \cite{Fri.Yom3} depends {\it only on the number of the removed points $\{u_1,\dots,u_d\}$, but not on their specific position}.

\smallskip

Having this covering in $L$, on top of each disk $D^j$ we build a complex ellipsoid $E^j$ which is an image of the unit ball $B_1$ under a complex linear mapping $\psi_j$. In order to ensure that the mapping $\psi_j$ is a doubling chart in $Y=\C^n\setminus H$ we have to ensure that the ellipsoid $4E^j$ also doesn't touch $H$, i.e. $4E^j\cap H = \emptyset$. This is done via comparing the distances of any point $v\in L$ to $H$, to $H\cap L$, and the value $|P(v)|$, respectively, which, in turn, is based on certain ``Remez-type'' inequalities (see \cite{Fri.Yom2}). Of course, what we get, is a uniform (and sufficiently accurate for our purposes) specific version of \L ojasiewicz inequality (see \cite{Kol,Sha.Kol.Shi} and references therein).

\smallskip

Then the chain joining $v_1$ and $v_2$ is constructed by following a certain continuous path $\gamma$ joining $v_1$ and $v_2$ in $Q^{\delta'}\cap L$. Since the disks $D^j$ form a covering of $Q^{\delta'}\cap L$, the subsequent disks $D^{j_s}$ along $\gamma$ form a chain with non-empty intersections from $v_1$ to $v_2$. The same remains true for the complex ellipsoids $E^j$. It may happen that for certain positions of $v_1,v_2$ with respect to $H$ the restriction of the polynomial $P$ to $L$ goes ``near-degenerates''. In order to avoid such situations, and to get the bounds in terms of $n,d$ only, we have to adapt the following strategy: instead of joining $v_1,v_2$ directly as we described, we join them using an auxiliary ball $B_\rho\subset Q^{\delta}$ of a sufficiently large radius $\rho$ (depending on $n,d$ only), so that $4B_\rho\subset \C^n\setminus H$. The existence of such ball is provided by Vitushkin's bounds.

\smallskip

We choose a point $z_1\in B_\rho$ so that for a complex straight line $L_1$ passing through $v_1$ and $z_1$, the restriction of the polynomial $P$ to $L_1$ has the norm, explicitly bounded from below. We repeat this argument for $v_2$, finding $z_2\in B_\rho$ and a complex straight line $L_2$ with the above properties (resp.). We find a chain $Ch_1$ joining $v_1$ with $z_1\in B_\rho$, and a chain $Ch_2$ joining $v_2$ with $z_2\in B_\rho$. Finally, the chain joining $v_1$ and $v_2$ is the union of the charts in $Ch_1$, $Ch_2$, and of the ball $B_\rho$ itself, which is a doubling chart in $Y$.

\begin{rem}
Our result concerns the hypersurface $H$ itself, and not its representation $H=\{P=0\}$, for a specific polynomial $P$. So writing $P=\prod_{i=1}^m P_i^{q_i}$ as a product of powers of prime polynomials $P_i$ we can assume that all $q_i=1$. Consequently, all the roots of $P_L$, which is $P$ restricted to a generic line $L\subset \C^n$, are simple. Hence, the discriminant $\Delta(P_L)$ of $P_L$ is generically non-zero, and this fact may allows us to estimate the minimal distance between the roots of $P_L$, which, in turn, implies some information on connected components of $Q\setminus H^\delta$. However, putting into consideration the discriminant $\Delta(P_L)$ of $P_L$, we obtain bounds depending on the specific polynomial $P$, {\it and not only on its degree $d$}, while the goal of the present paper is to get bounds on the doubling chains in $Q\setminus H^\delta$ in terms of $d$ and $\delta$ only.
\end{rem}

\subsection{Connectivity of the complement $Q^\delta\cap L$} \label{sec-connect}

Let $D_1$ be the unit disk in $\C$, and let $Z=\{z_1,\ldots,z_d\}$ be a finite set of points in $D_1$, and consider the set $D_1^\delta=D_1\setminus Z^\delta$. This set may have several connected components, and the picture strongly depends on the size of $\delta$ (see Figure $1$ below). Clearly, for $\delta$ small enough (e.g. half of the minimal distance between the points in $Z$) the complement $D_1^\delta$ is connected. The following lemma shows that for $\delta$ and $\delta'=\frac{\delta}{10d}$, depending only on $d$, but not on the mutual position of the points in $Z$, any two points in $D_1^\delta$ can be connected in $D_1^{\delta'}$.

\begin{lem} \label{lem-connectivity}
Let $D_1$ be the unit disk in $\C$, and let $Z=\{z_1,\ldots,z_d\}$ be a finite set of points in $D_1$. Then for any $0<\delta<1/2d$ and for $\delta'=\frac{\delta}{10d}$, any two points $v_1,v_2\in D_1\setminus Z^{\delta}$ belong to the same connected component of $D_1\setminus Z^{\delta'}$.
\end{lem}

\begin{proof}
Consider the set $U=D_1\cap Z^{\delta'}$, and let $U_i$, $i=1,\dots,l$ be all its the connected components. Notice that the diameter of each connected component $U_i$ does not exceed $2d\delta'$. Indeed, since $U_i$ is connected, it is a union of a ``connected tree'' of disks $D_{\delta'}^{i_q}$ around certain points $z_{i_q}\in Z$. The intersections of the neighboring disks in this tree are not empty, that is the distance between their centers is at most $2\delta'$. The number of the disks in each $U_i$ does not exceed $d$, and therefore, the diameter of $U_i$ does not exceed $2d\delta'$.

\smallskip

Next we form (possibly larger) components $\tilde U_i\supseteq U_i$, taking into account the possible intersections of $U_i$ with the unit circle $S_1$, which is the boundary of the disk $D_1$. For this purpose we consider each disk $D_{\delta'}^{i_q}$ in $U_i$ which touches $S_1$, and mark the point $y_{i_q}$ on $S_1$ which is the radial projection of the center $z_{i_q}$ of $D_{\delta'}^{i_q}$. To form $\tilde U_i$ we add to $U_i$ small neighborhoods of the arcs in $S_1$ joining the subsequent points $y_{i_q}$ on $S_1$ (see Figure 1). Notice that since the diameter of each $U_i$ does not exceed $2d\delta'=\delta/5< \frac{1}{10d}$, these joining arks are uniquely defined, and their total length does not exceed $2d\pi\delta'$. Thus, the diameter of $\tilde U_i$ does not exceed $2d\pi\delta'+2d\delta'<10d\delta'=\delta$.

\smallskip

Our next step is to form (possibly larger) simply-connected domains $\bar U_i\supseteq \tilde U_i\supseteq U_i$ from $\tilde U_i$, ``filling in'' the possible holes in $\tilde U_i$. Denote by $\bar U$ the union of the non-intersecting simply-connected domains $\bar U_i$. We conclude that the complement $\Omega=D_1\setminus \bar U$ of $\bar U$ in the disk $D_1$ is connected. Indeed, all the $\bar U_i$ can be retracted to points by a family of ambient homeomorphisms of $D_1$.

\smallskip

Next we notice that the diameter of each $\bar U_i$ is preserved under the ``filling in'' operation, since for any added point $y_1$ in $\bar U_i$ its distance to any point $y_2$ in $\bar U_i$ is not larger than the distance between the ends $w_1,w_2$ of the straight segment in $\bar U_i$ through $y_1,y_2$. But these ends $w_1,w_2$ belong to $U_i$. Hence the diameter of each $\bar U_i$ is at most $\delta$. Since each $\bar U_i$ contains at least one point of $Z$, we conclude that each component $\bar U_i$ is completely contained in the $\delta$-neighborhood $Z^\delta$ of $Z$.

\smallskip

Finally, by the assumptions, $v_1,v_2\in D_1\setminus Z^{\delta}$, and therefore $v_1,v_2\in \Omega=D_1\setminus \bar U$. It was shown above that $\Omega$ is connected, and it is contained in $D_1\setminus Z^{\delta'}$. Therefore $v_1,v_2$ belong to the same connected component of $D_1\setminus Z^{\delta'}$. This completes the proof of Lemma \ref{lem-connectivity}.
\end{proof}

\begin{figure} \label{figure}
\includegraphics[scale=0.3]{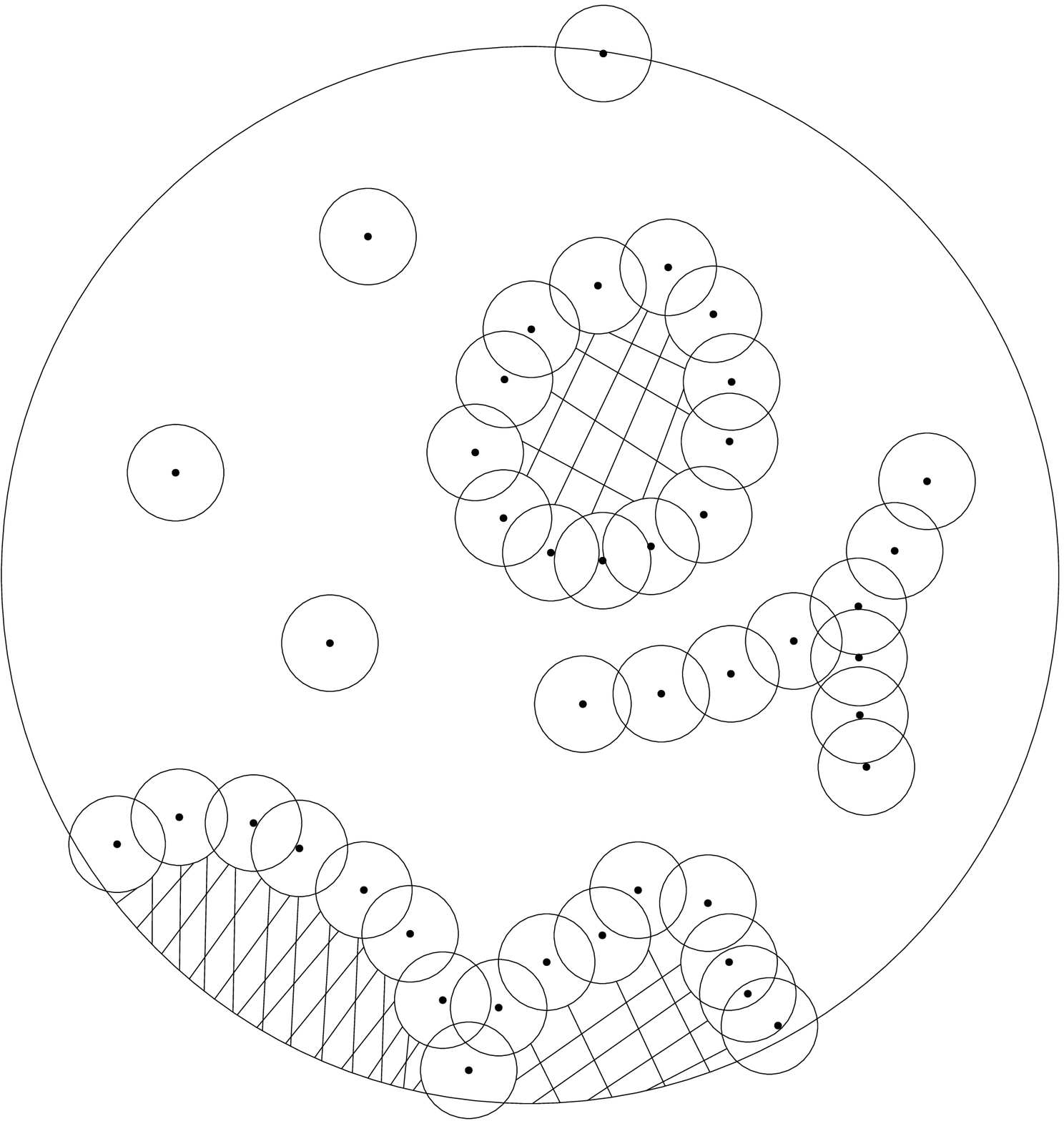}
\caption{}
\end{figure}

\begin{rem}
Notice that the initial value of $\delta$ should be chosen so that $D_1^\delta$ is not empty. More precisely, in view of the Proposition \ref{prp-B_rho} (below) we'll take $\delta<\rho(n,d)$ which still fits to our program where all constants depend only on $n,d$.
\end{rem}

\subsection{Balls in the complement of $H$} \label{sec-ball}

In this section we show that there is a ball $B_\rho\subset Q\setminus H$ of a radius $\rho=\rho(n,d)=\frac{1}{4(16(d+n))^n}$ (depending only on $n,d$) so that $4B_\rho\subset \C^n\setminus H$. In particular, under the assumptions of Theorem \ref{thm-main} we have $B_\rho\subset Q^{\delta}\subset Q^{\bar\delta}$, and we can join $v_1,v_2$ through any points in $B_\rho$.

%SHOULD BE PART OF ANOTHER SECTION PROBABLY
%Below we find, for given $v_1,v_2\in Q^\delta$ the points $z_1,z_2 \in B_\rho$ such the norm of the restriction of $P$ to the complex straight line $L$ passing through $v_1$ and $z_1$ (or through $v_2$ and $z_2$) is explicitly bounded from below. A similar construction can be used in many other situations. In particular, assuming that the discriminant of the restriction of $P$ to the complex straight lines $L$ is nonzero, we can find $z\in B_\rho$ so that the discriminant of the restriction o{prp-B_rho}f $P$ to the complex straight line $L$ passing through $v_1$ and $z_1$ (or through $v_2$ and $z_2$) is explicitly bounded from below. A similar approach was recently used in \cite{Bin.Nov} in a construction of a ``Weierstrass parametrization''.

\begin{prp} \label{prp-B_rho}
Let $P$ be a polynomial of degree $d$ in $\C^n$, and let $H=\{P=0\}$. Then, there exists a ball $B_\rho\subset Q^{\rho}$ of radius $\rho=\rho(n,d)=\frac{1}{4(16(d+n))^n}$.
\end{prp}

\begin{proof}
Let $\e>0$. We subdivide the unit cube $Q = Q_1 = [0,1]^{2n}$ into adjacent $\e$-cubes $Q_\e$. We use Vitushkin's bound on the metric entropy of algebraic sets (see \cite{Vit,Yom.Com}) in the form presented in \cite{Fri.Yom1} to show that for a certain $\e>0$ there exists a sub-cube $Q_\e$ which doesn't touch $H$.

\smallskip

Indeed, taking into account that $H=\{P=0\}$ is defined by two real polynomial equations
$$
g_1(z) = {\rm Re} P(z) = 0 \quad \text{and} \quad g_2(z) = {\rm Im} P(z) = 0
$$
of degree $d$ in $\C^n$ (we identify $\C^n$ with $\R^{2n}$), we conclude that $H$ has a real dimension $2n-2$. Thus, \cite[Theorem 1]{Fri.Yom1} yields that
$$
M(\e,H) \le C_0+C_1/\e+C_2/\e^2+\dots+C_{2n-2}/\e^{2n-2}
$$
where $C_s = \hat C_{2n-2-s}2^s \binom{2n}{s}$, $\hat C_{\ell}$ is the maximal numbers of connected components of $H\cap Q\cap W_{\ell}$, and $W_{\ell}$ is a real $\ell$-dimensional affine subspace of $\C^n$.

\smallskip

The constants $\hat C_{\ell}$ can be bounded via the standard bounds on the number of connected components of semi-algebraic sets (see e.g. \cite{Yom.Com} and references therein), which in turn, are estimated via the corresponding diagrams (see \cite[Definition 4.2]{Yom.Com}): the semi-algebraic set $H\cap Q$ is defined by two real polynomial equations of degree $d$ (or by $4$ inequalities $\pm g_i \ge 0$, $i = 1,2$), and by $4n$ real linear inequalities, defining $Q$. In turn, the intersection $H\cap Q\cap W_{\ell}$ is defined by the same inequalities as above in $W_{\ell} \cong \R^{\ell}$. Accordingly, the diagram of $H\cap Q$ is $\{2n,1,4n+4,d,d,1,\dots,1\}$, while the diagram of $H\cap Q\cap W_{\ell}$ is $\{\ell,1,4n+4,d,d,1,\dots,1\}$. In particular, the bound of \cite[Theorem 4.8]{Yom.Com} gives
$$
\hat C_{\ell} \le 2(d+n)(4d+4n-1)^{\ell-1} \le (4(d+n))^{\ell}
$$
which implies $C_s = \hat C_{2n-2-s}2^s \binom{2n}{s} \le (8(d+n))^{2n} \binom{2n}{s}$. Thus, the number $M(\e,H)$ of the subdivision cubes $Q_\e$ in $Q$ which have a non-empty intersection with $H$ satisfies
$$
M(\e,H) \le \frac{1}{\e^{2n-2}} \sum_{s = 0}^{2n-2} C_s \le \frac{(8(d+n))^{2n}}{\e^{2n-2}} \sum_{s = 0}^{2n-2} \binom{2n}{s} \le \frac{(16(d+n))^{2n}}{\e^{2n-2}} .
$$

We choose $\e>0$ so that $M(\e,H)$ is less than the total number $1/\e^{2n}$ of the sub-cubes $Q_\e$ in $Q$, to ensure the existence of a sub-cube outside of $H$,
$$
M(\e,H) \le (16(d+n))^{2n}/\e^{2n-2} \le 1/\e^{2n} .
$$

We set $\e = 1/(16(d+n))^n$), and to complete the proof, we take a concentric cube of size $\frac{1}{2(16(d+n))^n}$, whose distance to $H$ is at least $\frac{1}{4(16(d+n))^n}$, and inscribe into it the required ball $B_\rho$ of radius $\rho=\rho(n,d)=\frac{1}{4(16(d+n))^n}$.
\end{proof}

\subsection{Norm of a polynomial restricted to a complex straight line} \label{sec-norm}

Denote by $S^{2n-1}$ the unit sphere in $\C^n$, and consider a complex straight line $L = \{vt +b: t\in \C\} \subset \C^n$, with $v\in S^{2n-1}, b\in Q$. We denote by $P_{b,v} = P_L$ the restriction of $P$ to the complex straight line $L$. Let $\Omega \subset S^{2n-1}$ be a measurable set, with $\mu(\Omega) > 0$, where $\mu$ denotes the normalized Lebesgue measure on $S^{2n-1}$. Starting with Proposition \ref{prp-norm} we denote by $c_{1}, c_{2},\dots$ the constants depending only on the dimension $n$ and degree $d$.

\begin{prp} \label{prp-norm}
Let $P$ be a normalized polynomial of degree $d$ in $\C^n$. Then the norm $\|P_L\|$ satisfies
\be \label{eq-norm}
c_{1} \mu(\Omega)^d \le \max_{v\in \Omega} \|P_L\| \le c_{2} .
\ee
\end{prp}

\begin{proof}
First, note that we may assume that $b=0$. Indeed, the norm of the shift operator to $b\in Q$ on the space of polynomials $P$ of degree $d$ on $\C^n$ is bounded by $c_{2}$, (as well as the norm of the inverse operator, which is the shift to $-b\in Q$). Therefore, shifting a normalized polynomial, we can assume that $b=0$, and
\be \label{eq-norm.shift}
\frac{1}{c_{2}}\le ||P||\le c_{2}.
\ee

Therefore, assuming that $b=0$, we have for $P_v:=P_L$
$$
P_{v}(t) = P(vt) = \sum_{j=1}^d\sum_{\alpha:|\alpha|=j} a_\alpha v^\alpha t^j + a_0 = \sum_{j = 1}^d P_j(v)t^j + a_0
$$
where $P_j(v) = \sum_{\alpha:|\alpha| = j}a_\alpha v^\alpha$, $j=1,\dots,d$ are homogeneous polynomials in $v$ of degree $j$. Thus,
$$
\|P_{v}\| = \sum_{j=1}^d |P_j(v)| + |a_0| \le \sum_{j=0}^d \sum_{\alpha:|\alpha|=j} |a_\alpha| = \|P\| \le c_{2},
$$
which proves the upper bound.

\smallskip

To prove the lower bound in \eqref{eq-norm} we fix an index $j_0$ for which the norm $\|P_{j_0}\| = \sum_{\alpha:|\alpha| = j_0}|a_\alpha|$ is maximal. In particular, by \eqref{eq-norm.shift} we have $\|P_{j_0}\| \ge \frac{1}{(d+1)c_{2}}$. Now, we use (a version of) the Remez inequality for homogeneous polynomials of degree $j_0$ on $S^{2n-1}$ (see \cite{Bru.Gan}):
$$
\|P_{j_0}\| \le \frac{c_{7}}{\mu(\Omega)^{j_0}} \max_{v\in \Omega} |P_{j_0}(v)| \le \frac{c_{7}}{\mu(\Omega)^d} \max_{v\in \Omega} |P_{j_0}(v)|.
$$

Therefore,
\begin{align*}
\max_{v\in \Omega} \|P_v\| &= \max_{v\in \Omega} \sum_{j = 1}^d|P_j(v)| + |a_0| \ge \max_{v\in \Omega} |P_{j_0}(v)| \ge \frac{\mu(\Omega)^d}{c_{3}} \|P_{j_0}\| \\
&\ge \frac{\mu(\Omega)^d}{c_{7}c_{2}(d+1)}=: c_{1} \mu(\Omega)^d,
\end{align*}
which completes the proof of Proposition \ref{prp-norm}.
\end{proof}

Put $\rho=\rho(n,d)$, and let $B_\rho$ be the ball constructed in Proposition \ref{prp-B_rho}. We consider the complex straight line $L$ passing through the points $v_1$ and $z\in B_\rho$. Proposition \ref{prp-norm} implies the following:

\begin{cor} \label{cor-z_1}
There exists a point $z_1\in B_\rho$ so that for the complex straight line $L$ passing through the points $v_1$ and $z_1$, the norm of the restriction $P_L$ satisfies
$$
\|P_L\| \ge c_{3}.
$$
\end{cor}

\begin{proof}
Define $\Omega \subset S^{2n-1}$ as the set of all the vectors $v = \frac{z-v_1}{\|z-v_1\|}$ for $z\in B_\rho$. Note that $\mu(\Omega)$ is (up to a constant depending only on $n$) at least the ratio of the volume of the sphere of radius $\rho$ and the volume of the unit sphere $S^{2n-1}$, which is $\rho(n,d)^{2n-1}$. Therefore, we have $\mu(\Omega)\ge c_{8}$. Now, fix $z_1\in B_\rho$ for which the maximum of the norm $\|P_L\|$ is achieved. Thus, by Proposition \ref{prp-norm} we have $\|P_L\| \ge c_{1} \mu(\Omega)^d \ge c_{1} c_{8}^d=:c_{3}$, which completes the proof of Corollary \ref{cor-z_1}.
\end{proof}

\subsection{Comparing distance to a complex hypersurface and to its line section} \label{sec-chart}

For a point $v$ in a complex straight line $L$ we prove that the following quantities: $\dist (v,H)$, $\dist (v,H\cap L)$ and $|P_L(v)|$ are comparable.

\begin{prp} \label{prp-dist-compare}
Let $P$ be a normalized polynomial of degree $d$ in $\C^n$, and let $H = \{P = 0\}$. Let $L$ be a complex straight line in $\C^n$, and let $v\in L\cap Q$. Then
$$
c_{4} \|P_L\| \dist (v,H\cap L)^d \le \dist (v,H)\le \dist (v,H\cap L).
$$
\end{prp}

\begin{proof}
First, note that the right hand side inequality is obvious. For the lower bound, the proof is based on a comparison of both the distances $\dist (v,H)$, and $\dist (v,H\cap L)$, with the value $|P_L(v)|$ of the polynomial $P$ restricted to $L$. We start with a simple bound based on Markov's inequality:

\begin{lem} \label{lem:dist.P}
For any $v\in Q$ we have
$$
|P(v)| \le nd^22^{d} \dist (v,H).
$$
\end{lem}

\begin{proof}
Let us consider a twice larger concentric cube $2Q$. For any $z\in Q$ we have
$$
|P(z)| \le \sum_{\alpha:|\alpha| \le d}|a_\alpha| |z|^\alpha \le \|P\| = 1.
$$

Therefore, $\max_{Q} |P(z)| \le 1$, and by Markov's inequality (see, for instance, \cite{Ach}) we conclude that $\|\nabla P(z)\| \le nd^2$ for any $z\in Q$. Hence, by Bernstein (or Remez) inequality, we have $\|\nabla P(z)\| \le 2^dnd^2$ for any $z\in 2Q$. Let $u$ be the closest point to $v$ in $H = \{P = 0\}$. Then, assuming that $H\cap Q\ne \emptyset$, we have $u \in 2Q$. Integrating along the segment $[u,v]$ we obtain $|P_L(v)|=|P(v)| \le 2^{d}nd^2 \|u-v\|=nd^22^{d} \dist (v,H)$. This proves Lemma \ref{lem:dist.P}.
\end{proof}

Let $L= \{z = wt+v\}$ be a complex straight line with $w\in S^{2n-1}$. Consider the univariate polynomial in $t\in \C$
$$
p(t) = P(wt+v) = P_L(wt+v)
$$
and denote by $t_1,\dots,t_d$ all its roots. Thus, the points $u_s = wt_s+v$ are exactly the points of the intersection $H\cap L$, which implies that
$$
\eta := \dist(v,H\cap L) = \min_{s = 1,\dots,d} |t_s|.
$$

Now, we want to show that $P(v)=P_L(v)=p(0)$ is ``big'' in comparison with $\dist (v,H\cap L)$. This is a general fact about univariate polynomials:

\begin{lem} \label{lem-one.dim.bd}
Let $p(t)$ be a univariate complex polynomial of degree $d$, and let $Z=\{t_1,\ldots,t_d\}$ be its set of zeroes. Then for any $v\in \C$ we have
$$
|p(v)|\ge c_d ||p|| \dist (v,Z)^d
$$
where $c_d=\frac{1}{4(d+1)48^d}$.
\end{lem}

\begin{proof}
By the same reasoning as in the proof of Proposition \ref{prp-norm} above, it is enough to prove this inequality for $v=0\in \C$. Put $\eta := \dist(v,Z) = \min_{s = 1,\dots,d} |t_s|$. First, we consider a disk $D_{\eta/4} = \{|t| \le \eta/4\}$ in $L$, and show that $\max_{D_{\eta/4}}|p(t)|$ is big. For this purpose we apply the following polynomial doubling inequality, which is a special case of an extended Remez inequality for complex polynomials (see \cite[Theorem 4.1]{Fri.Yom2}).

\begin{lem}
Let $p$ be a univariate polynomial of degree $d$. Let $D_\kappa\subset D_1$ be a disk of radius $0<\kappa<1$, not necessarily concentric to $D_1$. Then
$$
\max_{D_{1}}|p(t)| \le (12/\kappa)^d \max_{D_{\kappa}}|p(t)|.
$$
\end{lem}

As a consequence, we obtain
$$
\max_{D_{\eta/4}}|p(t)| \ge (\eta/48)^d \max_{D_{1}}|p(t)| \ge \frac{\eta^d}{(d+1)48^d} \|p\|.
$$
Indeed, by Cauchy formula each coefficient of $p(t)$ is bounded by $\max_{D_{1}}|p(t)|$, and hence $\|p\| \le (d+1) \max_{D_{1}}|p(t)|$.

\smallskip

Now, we use the fact that all the roots of $p$ are outside of the disk $D_{\eta}$ in order to show that
$$
|p(0)| \ge 4^{-d} \max_{D_{\eta/4}}|p(t)|.
$$

Indeed, write $p(t)$ as the product $p(t) = \gamma\prod_{s = 1}^d(t-t_s)$, and notice that for any $t\in D_{\rho/4}$ and $s = 1,\dots,d$ we have $\frac{1}{2} \le \frac{|t-t_s|}{|t_s|} \le 2$. We conclude that for any two points $\tau_1,\tau_2\in D_{\eta/4}$ we have $\left|\frac{p(\tau_1)}{p(\tau_2)}\right| \le 4^d$ which implies that
$$
|p(0)| \ge \frac{\eta^d}{4(d+1)48^d} \|p\| =: c_d \eta^d \|p\| = c_d \|p\| \dist (v,Z)^d
$$
which proves Lemma \ref{lem-one.dim.bd}.
\end{proof}

Applying this result to the polynomial $P_L$, we get
$$
|P(v)|=|P_L(v)|\ge c_d ||P_L|| \dist(v,H\cap L)^d,
$$
and by Lemma \ref{lem:dist.P} we obtain that
\be\label{eq:dist.33}
\dist (v,H)\ge \frac{|P(v)|}{nd^22^{d}} \ge c_{4} ||P_L|| \dist(v,H\cap L)^d .
\ee

This concludes the proof of Proposition \ref{prp-dist-compare}.
\end{proof}

\begin{cor} \label{cor:dist.Value.compare}
Let $P$ be a normalized polynomial of degree $d$ in $\C^n$, and let $H = \{P = 0\}$. Then for any $v\in \C^n$ with $||v|| \le 1$ we have
$$
|P(v)| \ge c_{5} ||P_L|| \dist(v,H)^d,
$$
where $c_{5} = c_{4}nd^22^{d}$.
\end{cor}
\begin{proof}
This follows directly from the right hand side of the inequality (\ref{eq:dist.33}), since clearly we have
$$
\dist(v,H\cap L) \ge \dist(v,H).
$$
This corollary provides, of course, a certain specific global version of the \L ojasiewicz inequality (compare \cite{Sha.Kol.Shi,Kol}). 
\end{proof}

Now, we apply Corollary \ref{cor-z_1} and find a point $z_1\in B_\rho$ so that the norm of $P_L$ satisfies $\|P_L\| \ge c_{3}$ where $L$ is the complex straight line passing through $v_1$ and $z_1$. Next, we apply Proposition \ref{prp-dist-compare}, and conclude that for any $w\in L\cap Q$ we have

\begin{cor} \label{cor-dist-compare}
There exists a point $z_1\in B_\rho$ so that for the complex straight line $L$ passing through $v_1$ and $z_1$, and for any $w\in L$ we have
$$
\dist (w,H) \ge c_{6} \dist (w,H\cap L)^d .
$$
\end{cor}

\subsection{Construction of a chain} \label{sec-chain}

Now, we are ready to complete the construction of the chain $Ch$ joining $v_1$ and $v_2$. First, we apply Corollary \ref{cor-dist-compare} to find a point $z_1\in B_\rho$ so that
\be \label{eq-comp-dist}
\dist (v,H) \ge c_{6} \dist (v,H\cap L)^d
\ee
where $L:=L_{z_1}$ is the complex straight line passing through $v_1$ and $z_1$.

\smallskip

Let $Z = \{u_1,\dots,u_d\} = L\cap H$ be the zeros set of $P_L$. Consider the punctured disk $D_1^\delta = D_1\setminus Z^{\delta}$, where $D_1$ is the disk of radius $1$ centered at $v_1$ in $L$, and $Z^{\delta}$ is a $\delta$-neighborhood of $Z$ in $L$. As above, we assume that $\delta < \rho(n,d)$. Therefore, by Lemma \ref{lem-connectivity}, we conclude that for $\delta'=\frac{\delta}{10d}$, any two points in $D_1^\delta$ (in particular, $v_1$ and $z_1$) belong to the same connected component of $D_1\setminus Z^{\delta'}$.

\smallskip

The distance of $v_1$ and $z_1$ from $H$ is at least $\delta$, and for any $v\in L$ we have
$$
\delta \le \dist (v,H)\le \dist (v,H\cap L) = \dist (v,Z) .
$$

Thus, the distance of $v_1$ and $z_1$ in $L$ from $Z$ is also at least $\delta$, i.e. $v_1,z_1\in D_1^\delta$. Hence, both $v_1$ and $z_1$ belong to the same connected component of $D_1^{\delta'}$.

\smallskip

Now, we apply \cite[Theorem 2.2]{Fri.Yom3} to build a $\beta$-doubling covering $\U$ of $D_1^{\delta'}$, consisting of the disks $D^j$, with $\beta = 6$ (i.e. the 6 times larger concentric disks still do not touch $Z$), possessing the following properties:

\smallskip

\noindent 1) The number of disks $D^j\in\U$ is at most $18d\log(18/\delta')=18d\log(180d/\delta)$.

\smallskip

\noindent 2) If $D^i\cap D^j\ne \emptyset$, then the ratio of the radii of these disks may be only $\frac12,1,2$, and their intersection contains a disk of a radius at least $\frac{1}{3}$ of the smallest between radii $R_i, R_j$.

\smallskip

Next, on each $D^j\in\U$ of radius $R_j$ we build the ellipsoid $E^j$ with the rest of semi-axes equal to $c_{6} R_j^d/4$. Consider some orthonormal complex coordinates $(\phi_1,\dots,\phi_n)$, centered at $v_1$, so that the first axis $O\phi_1$ coincides with $L$. Then the $E^j$'s are the images of the unit ball $B_1\subset \C^n$ under the complex linear mapping $\psi_j$ with the diagonal matrix $A = \diag (R_j,{c_{6}R_j^d}/{4},\dots, {c_{6}R_j^d}/{4})$, with respect to the coordinates $(\phi_1,\dots,\phi_n)$.

\smallskip

Inequality \eqref{eq-comp-dist} above shows that the mappings $\psi_j$ are extendable to $B_4$ as mappings to $Y = \C^n\setminus H$. Indeed, for any point $z\in 4E^j$ consider its projection $\hat z$ to $L$. Clearly, $\hat z \in 4D^j$. Since the disk $6D^j$ does not touch $Z$, we conclude that $\dist (\hat z,H\cap L) \ge 2R_j$. Now, by \eqref{eq-comp-dist},
$$
\dist (\hat z,H) \ge c_{6}\dist (\hat z,H\cap L)^d \ge c_{6}(2R_j)^d.
$$
On the other hand, the distance $\|z-\hat z\|$ does not exceed the second semi-axis of $4E^j$, which is equal, by construction, to $c_{6} R_j^d$. So the distance of $z$ from $H$ is at least
$$
c_{6}(2R_j)^d-c_{6}R_j^d > 0.
$$

We estimate the intersection radius $\rho(E^i,E^j)$ in cases where this intersection is not empty. By the construction, $E^i\cap E^j\ne \emptyset$ if and only if $D^i\cap D^j\ne \emptyset$. In this last case the ratio of the radii of these disks may be only $\frac12,1,2$, and their intersection contains a disk of a radius at least $\frac{1}{3}$ of the smallest between $R_i$ and $R_j$. If the ratio $\frac{R_i}{R_j}$ is one, clearly $\rho(E^i,E^j) \ge \frac{1}{3}$. If the ratio $\frac{R_i}{R_j}$ is two, the height (over $L$) of the smaller ellipsoid $E^j$ is $2^d$ times smaller than the height of $E^i$. In this case it is easy to see that the preimage of $E^i\cap E^j$ under $\psi_j$ contains a ball of radius at least $\frac{1}{3}$, while the primage of $E^i\cap E^j$ under $\psi_i$ contains a ball of radius at least $2^{-d}/3$. Hence for any $E^i,E^j$ with a non-empty intersection we have $\rho(E^i,E^j) \ge 2^{-d}/3$.

\smallskip

Finally, we construct a chain $Ch_1$ joining $v_1$ and $z_1$, by following a certain continuous path $\gamma$ joining $v_1$ and $z_1$ in $D_1^{\delta'}$. Since the disks $D^j$ form a covering of $D_1^{\delta'}$, the subsequent disks $D^{j_s}$ along $\gamma$ (after omitting repetitions) form a chain with non-empty intersections in $D_1^{\delta'}$ from $v_1$ to $z_1$. The corresponding $E^{j_s}$ form the required chain $Ch_1$ in $Q^{\delta}$. The length of this chain does not exceed $18d\log(180d/\delta)$, the total number of the disks in $\U$. By repeating the above procedure for $v_2$ we construct a chain $Ch_2$ joining $v_2$ with another point $z_2\in B_\rho$. The chain joining $v_1$ and $v_2$ is the union of the charts in $Ch_1$, $Ch_2$, and of the ball $B_\rho$ itself, which is a doubling chart in $Y$. It contains at most $36d\log(180d/\delta)+1$ charts. This completes the proof of Theorem \ref{thm-main}.

\section{Doubling chains and Kobayashi metric} \label{sec-kob}

We present an upper bound on the Kobayashi distance between two points $p,q\in Y$. Let us recall the definitions. Let $Y$ be a complex $n$-dimensional manifold, and let $p, q\in Y$. The Kobayashi distance, or more accurately, pseudo-distance $d(p, q)$ is defined as follows (see \cite{Kob}). Choose points $p=p_0, p_1, \dots, p_{k-1}, p_k=q\in Y$, points $a_1, \dots, a_k$, $b_1, \dots, b_k$ in the unit disk $D_1\subset \C$, and holomorphic mappings $f_1, \dots, f_k:D_1 \to Y$, so that $f_i(a_i)=p_{i-1}$, $ f_i(b_i)=p_{i}$, $i=1, \dots, k$. Form a sum $\sum_{i=1}^k\rho(a_i, b_i)$, where $\rho$ is a Poincar\'e metric on $D_1$, and put $d(p, q)$ to be the infimum of these sums for all possible choices.

\smallskip

The following proposition (from \cite{Fri.Yom3}) shows that once we control the length of chains in doubling coverings $\U$, then it bounds the Kobayashi distance on $Y$.

\begin{prp}
Let $p,q\in Y$, and let $Ch$ be a doubling chain in $\U$ joining $p$ and $q$. Then the Kobayashi distance $d(p,q)$ satisfies $d(p,q) \le 3l(Ch)$.
\end{prp}

Thus, a direct application of the above proposition and Theorem \ref{thm-main} implies the following result:

\begin{cor} \label{thm-kobayashi}
Let $P$ be a polynomial of degree $d$ in $\C^n$, and let $H = \{P = 0\}$. Then for any $v_1,v_2\in Q^{\delta}$ the Kobayashi distance $d(v_1,v_2)$ in $Y$ does not exceed $180d\log(180d/\delta)$.
\end{cor}

\section{Doubling inequalities on complements of algebraic hypersurfaces} \label{sec-doubling} \label{sec-double}

In \cite{Fri.Yom3} it was shown that there is a very general explicit connection between doubling chains and doubling inequalities on $Y$ (Theorem 5.1 of \cite{Fri.Yom3}). Let $\Omega\subset G\subset Y$ be compact domains. For an analytic function $f$ in a neighborhood of $G$ in $Y$, the doubling constant of $f$ with respect to $\Omega$ and $G$ is the ratio $DC_f(G,\Omega) = {\max_{G}|f(z)|}/{\max_{\Omega}|f(z)|}$. Doubling inequalities provide an upper bound on this constant $DC_f(G,\Omega)$ for various classes of analytic functions $f$ on $Y$ (for more details on doubling inequalities see e.g. \cite{Bru,Fef.Nar,Fri.Yom3,Roy.Yom} and references therein).

\subsection{Doubling inequalities for algebraic functions}

We consider algebraic functions $y = g(z_1,\dots,z_{n})$ defined in $\C^n $ by an equation $Q(z,y) = 0$, $z = (z_1,\dots,z_{n})$. Here $Q(z,y)$ is a polynomial of degree $m$ in $\C^{n+1}$. The polynomial $Q(z,y) = \sum_{|\alpha|+j \le m} a_{\alpha,j}z^\alpha y^j$ can be written as a polynomial in $y$ with polynomial coefficients $P_j(z)$ in $z$:
$$
P(z,y) = \sum_{j = 0}^m P_j(z)y^j.
$$

Here $P_j(z) = \sum_{\alpha:|\alpha| \le m-j} a_{\alpha,j}z^\alpha$ for $j = 1,\dots,m$. The multivalued algebraic function $y = g(z)$ defined by equation $Q(z,y) = 0$ may have poles and ramification points, which are always contained in a certain hypersurface $\Sigma(g)$. Over $Y = \C^n\setminus \Sigma(g)$ the function $y = g(z)$ is a locally regular, but possibly multivalued analytic function. Our goal is to produce ``uniform'' (i.e. depending only on the degrees and on the distance to singularities) doubling inequalities for functions $g(z)$. To simplify the setting, we start with the hypersurface $H$, containing singularities of $g$: let $P$ be a polynomial of degree $d$ in $\C^n$, and let $H = \{P = 0\}$. We fix $\delta$ with $0<\delta < \rho(n,d)$. Now let $\Omega \subset G \subset Q^{\delta}$ be compact domains, with $G$ simply-connected. With respect to $\Omega$ we will assume in addition that it contains a ball $\bar B$ of a radius at least $\rho(n,d)/10$.

\begin{thm}\label{thm:doubling1}
%Let $P$ be a polynomial of degree $d$ in $\C^n$, and let $H = \{P = 0\}$. Let $\Omega \subset G \subset Q^{\delta}$ be compact domains, with $G$ simply-connected, and 
Let $g(z)$ be an algebraic function of degree $m$, regular over $Y = \C^n\setminus H$, and let $\tilde g$ a univalued branch of $g$ over $G$. Then
$$
DC_{\tilde g}(Q^{\delta},\Omega) \le C_1(\frac{1}{\delta})^{C_2d},
$$
with the constants $C_1,C_2$ depending only on $n,d,$ and $m$.
\end{thm}

\begin{proof}
Let $v_1$ be any point in $G$, and let $v_2$ be the center of the ball $\bar B \subset \Omega.$ By Theorem \ref{thm-main} there is a doubling chain $Ch$ joining $v_1\in G$ and $v_2\in \bar B \subset \Omega$ in $Q^{\delta}$, with $\rho(Ch) \ge \frac{1}{3} 2^{-d}$, and with the length
$$
l(Ch) \le 36d\log(180d/\delta)+1 .
$$
By a minor modification of the proof of Theorem \ref{thm-main}, we can assume that the last chart in $Ch$ is the ball $\bar B$ itself. Now a direct application of Theorem 5.1 of \cite{Fri.Yom3} completes the proof.
\end{proof}

Theorem \ref{thm:doubling1} provides a generalization of the doubling inequalities for algebraic functions $g$, obtained in \cite{Roy.Yom}, preserving the main feature of these results: the doubling constant depends only on the degree of $g$ and on the distance to its singularities.

\subsection{Lower bound on the length of chains}\label{Sec:lower.bd}

Now we can produce a lower bound for the length of the doubling chains in Theorem \ref{thm-main}: let $P$ be a polynomial of degree $d$ in $\C^n$, and let $H = \{P = 0\}$. As above, we put $Y = C^n\setminus H$, and, for a fix $\delta$ with $0<\delta < \rho(n,d)$, we denote by $H^\delta$ the $\delta$-neighborhood of $H$, and by $Q^\delta$ the complement $Q\setminus H^\delta$. Denote by $\hat B$ the ball of radius $\rho(n,d)$ inside $Q^\delta$, such that a four times larger concentric ball is still inside $Q^\delta$. In particular, the distance of $\hat B$ to $H$ is at least $2\rho(n,d)$. The existence of such a ball $\hat B$ was proved in Proposition \ref{prp-B_rho} above. 
 
\begin{thm}\label{thm:doubling2}
Let $v_1$ belongs to the boundary of $Q^\delta$ (i.e. the distance of $v_1$ to $H$ is $\delta$), and let $v_2\in \hat B$. Then for each doubling chain $Ch$ in $Y$, joining $v_1$ and $v_2$, and satisfying $\rho(Ch)\ge \eta$, we have
$$
l(Ch) \ge \frac{C_3\log (\frac{C_4}{\delta})}{\log (\frac{1}{\eta})}
$$
with the constants $C_3,C_4$ depending only on $n,d.$
\end{thm}

\begin{proof}
it is enough to present an algebraic functions $g$, and compact subsets $\Omega \subset G \subset Q^{\delta}$, with a ``big'' $DC_f(Q^{\delta},\Omega)$. We put $g=\frac{1}{P}$, and put $\Omega=\hat B$. By Lemma \ref{lem:dist.P} we have 
$$
|P(v_1)| \le nd^22^{d} \dist (v,H)=nd^22^{d} \delta.
$$

Therefore $|g(v_1)|=|\frac{1}{P(v_1)}|\ge \frac{1}{nd^22^{d}\delta}$. On the other hand, by Corollary \ref{cor:dist.Value.compare} we have $|P(v_2)| \ge c_{9}:= c_{5}(2\rho(n,d))^d$ on $\hat B$, since the distance of $v_2$ to $H$ is at least $2\rho(n,d)$, by construction of the ball $\hat B$. We conclude that $|g| \le c_{10}:= \frac{1}{c_{9}}$ on $\hat B$. Finally, the doubling constant $DC_{g}(Q^{\delta},\hat B)$ satisfies
$$
DC_{g}(Q^{\delta},\hat B) \ge c_{10} /\delta .
$$

A direct application of Theorem 5.1 (or Corollary 5.3) of \cite{Fri.Yom3} completes the proof of Theorem \ref{thm:doubling2}.
\end{proof}

\end{document}